\documentclass [11pt]{amsart}
\usepackage{amsmath,amssymb}
 \usepackage{amsthm, amsfonts}
 \usepackage{enumerate}
 \usepackage{amssymb,amsmath,amsthm, amsfonts}
 \usepackage{amsmath,amssymb}
 \usepackage{amsthm, amsfonts}
 \usepackage{enumerate}
 \usepackage{amssymb,amsmath,amsthm, amsfonts}
 \usepackage{cite}
\usepackage[pagewise]{lineno}
\newtheorem{theorem}{Theorem}[section]
\newtheorem{proposition}[theorem]{Proposition}
\newtheorem{lemma}[theorem]{Lemma}
\newtheorem{corollary}[theorem]{Corollary}
\theoremstyle{definition}
\newtheorem{definition}[theorem]{Definition}
\newtheorem{example}[theorem]{Example}

\theoremstyle{remark}

\numberwithin{equation}{section}

\usepackage{float}
\usepackage{graphicx}
\usepackage{caption}
\usepackage{subcaption}
\usepackage{hyperref}
\begin{document}

\title [{{Graded modules with Noetherian graded second spectrum}}]{Graded modules with Noetherian graded second spectrum}

 \author[{{S. Salam and K. Al-Zoubi,  }}]{\textit{ Saif Salam and Khaldoun Al-Zoubi}*}

\address
{\textit{ Saif Salam, Department of Mathematics and
Statistics, Jordan University of Science and Technology, P.O.Box
3030, Irbid 22110, Jordan.}}
\bigskip
{\email{\textit{smsalam19@sci.just.edu.jo}}}

\address
{\textit{Khaldoun Al-Zoubi, Department of Mathematics and
Statistics, Jordan University of Science and Technology, P.O.Box
3030, Irbid 22110, Jordan.}}
\bigskip
{\email{\textit{kfzoubi@just.edu.jo}}}


\date{}
\begin{abstract}
Let $R$ be a $G$ graded commutative ring and $M$ be a $G$-graded $R$-module. The set of all graded second submodules of $M$ is denoted by $Spec_G^s(M)$ and it is called the graded second spectrum of $M$. In this paper, we discuss graded rings with Noetherian graded prime spectrum and obtain some conclusions. In addition, we introduce the notion of the graded Zariski socle of graded submodules and explore their properties. Using these conclusions and properties, we also investigate $Spec_G^s(M)$ with the Zariski topology from the viewpoint of being a Noetherian space and give some related outcomes.
\end{abstract}

\keywords{ . \\
$*$ Corresponding author}
 \maketitle


 \section{Introduction and Preliminaries }
Let $G$ be an abelian group with identity $e$ and $R$ be a commutative ring with unity $1$. We say that $R$ is a $G$-graded ring if there exist a family $\{R_g\}_{g\in G}$ of additive subgroups of $R$ such that $R=\underset{g\in G}{\oplus} {R_g}$ and $R_g R_h\subseteq R_{gh}$ for all $g,h\in G$. The elements of the set $h(R)=\underset{g\in G}{\bigcup}R_g$ are called the homogeneous elements of $R$. In addition, the non-zero elements of $R_g$ are called homogeneous elements of degree $g$. Every element $x\in R$ can be written uniqely as $\underset{g\in G}{\sum}x_g$, where $x_g\in R_g$ and $x_g=0$ for all but finitely many $g$.  It is easy to see that if $R=\underset{g\in G}{\oplus} R_g$ is a $G$-graded ring, then $1\in R_e$ and $R_e$ is a subring of $R$, see \cite{nastasescu2004methods}. An ideal $I$ of a $G$-graded ring $R=\underset{g\in G}{\oplus}R_g$ is said to be a $G$-graded ideal of $R$, denoted by $I\lhd_G R$, if $I=\underset{g\in G}{\oplus}(I\cap R_g)$. The graded radical of the graded ideal $I$ is the set of all $a=\underset{g\in G}{\sum}a_g\in R$ such that for each $g\in G$ there exists $n_g > 0$ with $a_g ^{n_g}\in I$. By $Gr(I)$ (resp. $\sqrt{I}$) we mean the graded radical (resp. the radical) of $I$. Note that if $r\in h(R)$, then $r\in Gr(I)$ if and only if $r\in \sqrt{I}$. If $I=Gr(I)$, then we say that $I$ is a graded radical ideal of $R$, see \cite{refai2000graded}. The graded prime spectrum of $R$, given by $Spec_G(R)$, is defined to be the set of all graded prime ideals of $R$. For each graded ideal $J$ of $R$, define $V_G^R(J)$ as the set $\{p\in Spec_G(R)\,\mid\, J\subseteq p\}$. Then the collection $\{V_G^R(J)\,\mid\, J\lhd_G R\}$ satisfies the axioms for closed sets of a topology on $Spec_G(R)$. The resulting topology is called the Zariski topology on $Spec_G(R)$ (see, for example, \cite{refai2000graded,ozkiirisci2013graded,refai2001properties,al2020additional}).

Let $R$ be a $G$-graded ring and $M$ be a left $R$-module. We say that $M$ is a $G$-graded $R$-module if $M=\underset{g\in G}{\oplus}{M_g}$ and $R_g M_h\subseteq M_{gh}$ for all $g, h\in G$, where every $M_g$ is an additive subgroup of $M$. The elements of the set $h(M)=\underset{g\in G}{\bigcup} M_g$ are called the homogeneous elements of $M$. Also, the non-zero elements of $M_g$ are called homogeneous elements of degree $g$. If $t\in M$, then $t$ can be uniquely represented by $\underset{g\in G}{\sum}{t_g}$, where $t_g\in M_g$ and $t_g=0$ for all but finitely many $g$.  A submodule $N$ of a $G$-graded $R$-module $M=\underset{g\in G}{\oplus} M_g$ is said to be a $G$-graded submodule of $M$ if $N=\underset{g\in G}{\oplus}(N\cap M_g)$. By $N\leq _G M$, we mean that $N$ is a $G$-graded submodule of $M$. Let $M$ be a $G$-graded $R$-module, $I\lhd_G R$ and $N\leq_G M$. Then $Ann_R(N)=\{r\in R\,\mid\, rN=\{0\}\}\lhd_G R$ and $Ann_M(I)=\{m\in M\,\mid\, Im=\{0\}\}\leq_G M$, , see \cite{nastasescu2004methods}.

Let $M$ be a $G$-graded $R$-module. A non-zero graded submodule $S$ of $M$ is called graded second if $rS=S$ or $rS=0$ for every $r\in h(R)$. In this case, $Ann_R(S)$ is a graded prime ideal of $R$. The set of all graded second submodules of $M$ is denoted by $Spec_G^s(M)$ and it is known as the graded second spectrum of $M$. If $Spec_G^s(M)=\emptyset$, then we say that $M$ is a $G$-graded secondless $R$-module. The graded second radical (or graded second socle), $soc_G(N)$, of a $G$-graded submodule $N$ of $M$ is defined as the sum of all graded second submodules of $M$ contained in $N$. When $N$ does not contain graded second submodules, we set $soc_G(N)=\{0\}$. For more information about the graded second submodules and the graded second socle of graded submodules of graded modules (see, for example, \cite{ansari2012graded,cceken2015graded,salam2022zariski}).

Let $M$ be a $G$-graded $R$-module and  let $\Omega ^ {s*} (M)=\{V_G^{s*} (N)\, \mid \, N\leq_G M\}$ where $V_G^{s*}(N)=\{S\in Spec_G^s(M)\, \mid \, S\subseteq N\}$ for any $N\leq _G M$. We say that $M$ is a $G$-cotop module if the collection $\Omega^{s*}(M)$ is closed under finite union. When this is the case, $\Omega^{s*}(M)$ induces a topology on $Spec_G^s (M)$ having $\Omega^{s*}(M)$ as the collection of all closed sets and the generated topology is called the quasi-Zariski topology on $Spec_G^s(M)$. Unlike $\Omega^{s*}(M)$, $\Omega(M)=\{V_G^s(N)\, \mid \, N\leq_G M\}$ where $V_G^s(N)=\{S\in Spec_G^s(M)\, \mid \, Ann_R(N)\subseteq Ann_R(S)\}$ for any $N\leq_G M$ always satisfies the axioms for closed sets of a topology on $Spec_G^s(M)$. This topology is called the Zariski topology on $Spec_G^s(M)$. For a $G$-graded $R$-module $M$, the map $\phi:Spec_G^s(M)\rightarrow Spec_G(R/Ann_R(M))$ defined by $S\rightarrow Ann_R(S)/Ann_R(M)$ is called the natural map of $Spec_G^s(M)$. For more details concerning the topologies on $Spec_G^s(M)$ and the natural map of  $Spec_G^s(M)$, one can look in \cite{salam2022zariski}.

Let $M$ be a $G$-graded $R$-module. Then $M$ is said to be graded Noetherian (resp. graded Artinian) if it satisfies the ascending (resp. descending) chain condition for the graded submodules. The $G$-graded ring $R$ is said to be graded Noetherian (resp. graded Artinian) if it is graded Noetherian (resp. graded Artinian) as $G$-graded $R$-module, see \cite{nastasescu2004methods}. A topological space $X$ is Noetherian provided that the open (resp. closed) subsets of $X$ satisfy the ascending (resp. descending) chain condition, or the maximal (resp. minimal) condition, see \cite{atiyah1969introduction,bourbakialgebre}.

We start this work by studying graded rings with Noetherian graded prime spectrum and provide some related results which are important in the last section. For example, we show that every $G$-graded Noetherian ring $R$ has Noetherian graded prime spectrum which implies that every graded radical ideal $I$ of $R$ is the intersection of a finite number of minimal graded prime divisors of it (Proposition \ref{Proposition 2.2} and Theorem \ref{Theorem 2.7}). The notion of $RFG_g$-ideals will be introduced and some properties of them will be given. In Section 2 of this paper, among other important results, we prove that $Spec_G(R)$ for a $G$-graded ring $R$ is a Noetherian topological space if and only if every graded prime ideal of $R$ is an $RFG_g$-ideal (Corollary \ref{Corollary 2.13}). Let $M$ be a $G$-graded $R$-module. In Section 3, we say that $M$ is a graded secondful module if the natural map $\phi$ of $Spec_G^s(M)$ is surjective. The surjectivity of $\phi$ plays important role in this study. So we investigate some properties of the graded secondful modules. Then we define the graded Zariski socle, $Z.soc_G(N)$, of a graded submodule $N$ of $M$ to be the sum of all members of $V_G^s(N)$ if $V_G^s(N)\neq \emptyset$ and we put $Z.soc_G(N)=\{0\}$ otherwise. We provide some relationships between $V_G^{s*}(N)$ and $V_G^s(N)$ to obtain some situations where the graded second socle and the graded Zariski socle coincide (Lemma \ref{Lemma 3.5} and Proposition \ref{Proposition 3.6}). We also list some important properties of the graded Zariski socle of graded submodules (Proposition \ref{Proposition 3.8}). In Section 4, we study graded modules with Noetherian graded second spectrum and state some related observations. The properties of the graded Zariski socle of graded submodules are essential in this section. For example, we show that the graded second spectrum of a graded module is a Noetherian space if and only if the descinding chain condition for graded Zariski socle submodules hold (Theorem \ref{Theorem 4.1}). In addition, we show that the surjective relationship between $Spec_G^s(M)$ and $Spec_G(R/Ann_R(M))$ for a $G$-graded secondful $R$-module $M$ yields the characterization that $Spec_G^s(M)$ is a Noetherian space exactly if $Spec_G(R/Ann_R(M))$ is Noetherian (Theorem \ref{Theorem 4.5}). After this, we add some conditions on the graded secondful modules and prove that, under these conditions, every graded Zariski socle submodule is the sum of a finite number of graded second submodules (Theorem \ref{theorem 4.8}(1)). Next, we give the concept of $RFG^*_g$-submodules of graded modules. Using the results of Section 2 and Section 3, we prove the equivalence that a graded faithful secondful module has Noetherian graded second spectrum if and only if every graded second submodule is an $RFG^*_g$-submodule (Corollary \ref{corollary 4.12}(1)).

Throughout this paper, $G$ is an abelian group with identity $e$ and all rings are commutative with unity $1$. For a $G$-graded $R$-module $M$ and a graded ideal $I$ of $R$ containing $Ann_R(M)$, $I/Ann_R(M)$ and $R/Ann_R(M)$ will be expressed by $\overline{I}$ and $\overline{R}$, respectively.
\section{Graded rings with graded prime spectrum}
In this section, we give a few conditions under which the Zariski topology on $Spec_G(R)$ for a $G$-graded ring $R$ is a Noetherian space and obtain some related results that are needed in the last section.

Let $R$ be a $G$-graded ring and $Y$ be a subset of $Spec_G(R)$. Then the closure of $Y$ in $Spec_G(R)$ will be expressed by $Cl(Y)$. Also, the intersection of all members of $Y$ will be denoted by $\xi(Y)$. If $Y=\emptyset$, we write $\xi(Y)=R$. It is clear that $\xi(Y_2)\subseteq \xi(Y_1)$ for any $Y_1\subseteq Y_2\subseteq Spec_G(R)$.
\begin{lemma}\label{Lemma 2.1}
Let $R$ be a $G$-graded ring. Then we have the following:
\begin{enumerate}
\item If $Y\subseteq Spec_G(R)$, then $Cl(Y)=V_G^R(\xi(Y))$.
\item $\xi(V_G^R(I))=Gr(I)$ for each graded ideal $I$ of $R$.
\item $Gr(I_1)=Gr(I_2)\Leftrightarrow V_G^R(I_1)=V_G^R(I_2)$ for each graded ideals $I_1$ and $I_2$ of $R$.
\item If $R$ has Noetherian graded prime spectrum, then so does $R/I$ for any graded ideal $I$ of $R$.
\end{enumerate}
\end{lemma}
\begin{proof}
(1) Clearly, $Y\subseteq V_G^R(\xi(Y))$. For the reverse inclusion, let $V_G^R(I)$ be any closed subset of $Spec_G(R)$ containing $Y$, where $I\lhd_G R$. It is enough to show that $V_G^R(\xi(Y))\subseteq V_G^R(I)$. So let $p\in V_G^R(\xi(Y))$. Then $\xi(Y)\subseteq p$. Note that for any $p^\prime\in Y$, we have $I\subseteq p^\prime$ and thus $I\subseteq \underset{p^\prime\in Y}{\bigcap}p^\prime= \xi(Y)\subseteq p$. This implies that $p\in V_G^R(I)$. Therefore $V_G^R(\xi(Y))$ is the smallest closed set containing $Y$ and hence $Cl(Y)=V_G^R(\xi(Y))$. \\
(2) follows from \cite[Proposition 2.5]{refai2000graded}.\\
(3) Suppose that $Gr(I_1)=Gr(I_2)$. By (2), we have $\xi(V_G^R(I_1))=\xi(V_G^R(I_2))$ and this implies that $V_G^R(I_1)=Cl(V_G^R(I_1))=V_G^R(\xi(V_G^R(I_1)))=V_G^R(\xi(V_G^R(I_2)))=Cl(V_G^R(I_2))=V_G^R(I_2)$ by (1). The converse is clear by using (2) again.  \\
(4) Let $V_G^{R/I}(J_1/I)\supseteq V_G^{R/I}(J_2/I)\supseteq ...$ be a descending chain of closed sets in $Spec_G(R/I)$, where $J_i\lhd_G R$ containing $I$. Then it is easy to see that we obtain $V_G^R(J_1)\supseteq V_G^R(J_2)\supseteq ...$ which is a descending chain of closed sets in $Spec_G(R)$. By hypothesis, there exists $k\in \mathbb{Z}^+$ such that $V_G^R(J_i)=V_G^R(J_k)$ for each $i\geq k$. Therefore $V_G^{R/I}(J_i/I)=V_G^{R/I}(J_k/I)$ for each $i\geq k$, as desired.
\end{proof}
\begin{proposition} \label{Proposition 2.2} Let $R$ be a $G$-graded ring. Then the following hold:
\begin{enumerate}
\item[(a)] $R$ has Noetherian graded prime spectrum if and only if the ascending chain condition for graded radical ideals of $R$  hold.
\item[(b)] If $R$ is a graded Noetherian ring, then $Spec_G(R)$ is a Noetherian space.
\end{enumerate}
\end{proposition}
\begin{proof} (a) Suppose that $Spec_G(R)$ is a Noetherian space and let $I_1\subseteq I_2\subseteq ...$ be an ascending chain of graded radical ideals of $R$. By \cite[Proposition 3.1(1)]{refai2000graded}, we have $V_G^R(I_1)\supseteq V_G^R(I_2)\supseteq ...$ which is a descending chain of closed sets in $Spec_G(R)$. So $\exists k\in \mathbb{Z}^+$ such that $V_G^R(I_t)=V_G^R(I_k)$ for each $t\geq k$. By Lemma \ref{Lemma 2.1}(3), we obtain $I_t=Gr(I_t)=Gr(I_k)=I_k$ for each $t\geq k$ which completes the proof of the first direction. Now, suppose that the ascending chain condition for graded radical ideals hold and let $V_G^R(I_1)\supseteq V_G^R(I_2)\supseteq ...$ be a descending chain of closed sets in $Spec_G(R)$. Using Lemma \ref{Lemma 2.1}(2) and \cite[Proposition 2.4]{refai2000graded}, we have  $Gr(I_1)\subseteq Gr(I_2)\subseteq ...$ is an ascending chain of graded radical ideals of $R$ and hence $\exists k\in \mathbb{Z}^+$ such that $Gr(I_k)=Gr(I_t)$ for each $t\geq k$. Again, by Lemma \ref{Lemma 2.1}(3), we get $V_G^R(I_k)=V_G^R(I_t)$ for each $t\geq k$, as desired. Now, the proof of (b) is trivial by (a).
\end{proof}
 Let $W$ be a topological space. Then $W$ is said to be irreducible if $W\neq \emptyset$ and whenever $W_1$ and $W_2$ are closed subsets in $W$ with $W=W_1\cup W_2$, then either $W=W_1$ or $W=W_2$. Let $W^\prime\subset W$. Then $W^\prime$ is irreducible if it is an irreducible space with the relative topology. The maximal irreducible subsets of $W$ are called the irreducible components of $W$. It is easy to see that every singleton subset of $W$ is irreducible and a subset $Z$ of $W$ is irreducible if and only if its closure is irreducible (see, for example, \cite{atiyah1969introduction,bourbakialgebre}).
 \begin{theorem} \label{Theorem 2.3}
 Let $R$ be a $G$-graded ring and $Y\subseteq Spec_G(R)$. Then $Y$ is irreducible closed subset of $Spec_G(R)$ if and only if $Y=V_G^R(I)$ for some $I\in Spec_G(R)$.
 \end{theorem}
 \begin{proof}
 $\Rightarrow$: Since $Y$ is irreducible in $Spec_G(R)$, then $\xi(Y)\in Spec_G(R)$ by \cite[Lemma 4.3]{salam2021zariski}. Since $Y$ is closed, then $Y=Cl(Y)=V_G^R(\xi(Y))$ by Lemma \ref{Lemma 2.1}(1). Choose $I=\xi(Y)$. \\
 $\Leftarrow$: Suppose that $Y=V_G^R(I)$ for some $I\in Spec_G(R)$. By Lemma \ref{Lemma 2.1}(1), $Y=Cl(\{I\})$. Since $\{I\}$ is irreducible, then its closure $Cl(\{I\})=Y$ is irreducible, as needed.
 \end{proof}
 \begin{definition}\cite{refai2004graded}
 Let $R$ be a $G$-graded ring and $I\lhd_G R$. The graded prime ideal that is minimal with respect to containing $I$ is called minimal graded prime divisor of $I$ (or minimal graded prime ideal over $I$). That is, $p$ is a minimal graded prime divisor of $I$ if $p\in V_G^R(I)$ and whenever $J\in Spec_G(R)$ with $I\subseteq J\subseteq p$, we have $p=J$.
 \end{definition}
\begin{theorem} \label{theorem 2.5}
Let $R$ be a $G$-graded ring and $I\lhd_G R$. Then the following hold:
\begin{enumerate}
\item If $Y\subseteq V_G^R(I)$, then $Y$ is an irreducible component of the subspace $V_G^R(I)$ if and only if $Y=V_G^R(p)$ for some minimal graded prime divisor of $I$.
\item If $V_G^R(I)$ is a Noetherian topological subspace of $Spec_G(R)$, then $I$ contains only a finite number of minimal graded prime divisors.
\end{enumerate}
\end{theorem}
\begin{proof}(1) $\Rightarrow$: Suppose that $Y$ is an irreducible component of the subspace $V_G^R(I)$. By \cite[p. 13, Exercise 20(iii)]{atiyah1969introduction}, $Y$ is an irreducible closed in the subspace $V_G^R(I)$ and hence it is irreducible closed in $Spec_G(R)$. By Theorem \ref{Theorem 2.3}, $Y=V_G^R(p)$ for some $p\in Spec_G(R)$. It is clear that $p\in V_G^R(I)$. Let $J\in Spec_G(R)$ with $I\subseteq J\subseteq p$. Again, by Theorem \ref{Theorem 2.3}, $V_G^R(J)$ is irreducible in $Spec_G(R)$. But $V_G^R(J)\subseteq V_G^R(I)$ and $V_G^R(I)$ is closed in $Spec_G(R)$. This follows that $V_G^R(J)$ is irreducible in the subspace $V_G^R(I)$. Since $V_G^R(p)\subseteq V_G^R(J)\subseteq V_G^R(I)$ and $V_G^R(p)$ is irreducible component of $V_G^R(I)$, then $V_G^R(p)=V_G^R(J)$ and hence $p\subseteq J$. This shows that $p$ is a minimal graded prime divisor of $I$. \\
$\Leftarrow$: Suppose that $Y=V_G^R(p)$ for some minimal graded prime divisor $p$ of $I$. By Theorem \ref{Theorem 2.3}, $Y$ is irreducible in $Spec_G(R)$. Now, it is easy to see that $Y$ is irreducible in the subspace $V_G^R(I)$. Let $Y^\prime$ be irreducible in the subspace $V_G^R(I)$ with $Y\subseteq Y^\prime$. To complete the proof, it is enough to show that $Y^\prime\subseteq Y$. By Lemma \ref{Lemma 2.1}(2), we obtain $I\subseteq Gr(I)=\xi(V_G^R(I))\subseteq \xi(Y^\prime)\subseteq \xi(V_G^R(p))=Gr(p)$. Since $p\in Spec_G(R)$, then $Gr(p)=p$ by \cite[Proposition 2.4(5)]{refai2000graded}. Also, by \cite[Lemma 4.3]{salam2021zariski}, $\xi(Y^\prime)\in Spec_G(R)$. Since $I\subseteq \xi(Y^\prime)\subseteq p$ and $p$ is minimal graded prime ideal over $I$, then $\xi(Y^\prime)=p$. By Lemma \ref{Lemma 2.1}(1), $Y^\prime\subseteq Cl(Y^\prime)=V_G^R(\xi(Y^\prime))=V_G^R(p)=Y$, as desired.\\
(2) By \cite[P. 124, Proposition 10]{bourbakialgebre}, the subspace $V_G^R(I)$ has only finitely many irreducible components as $V_G^R(I)$ is Noetherian space. Now the result follows by (1).
\end{proof}
Let $R$ be a $G$-graded ring and $p\lhd_G R$. It is clear that $p$ is a minimal graded prime ideal of $R$ if and only if $p$ is a minimal graded prime divisor of $\{0\}$. Now, the following result can be easily checked by replacing $I$ by $\{0\}$ in Theorem \ref{theorem 2.5}.
\begin{corollary} The following hold for any $G$-graded ring $R$:
\begin{enumerate}
\item[(i)] If $Y\subseteq Spec_G(R)$, then $Y$ is an irreducible component of $Spec_G(R)$ if and only if $Y=V_G^R(p)$ for some minimal graded prime ideal $p$ of $R$.
\item[(ii)] If $Spec_G(R)$ is a Noetherian space, then $R$ contains only finitely many minimal graded prime ideals. Hence, by Proposition \ref{Proposition 2.2}, every graded Noetherian ring has only finitely many minimal graded prime ideals.
\end{enumerate}
\end{corollary}
\begin{theorem} \label{Theorem 2.7}
Let $R$ be a $G$-graded ring with Noetherian graded prime spectrum. Then every graded radical ideal $I$ of $R$ is the intersection of a finite number of minimal graded prime divisors of it.
\end{theorem}
\begin{proof} Note that $R$ is the intersection of the empty family of graded prime divisors of $R$. Suppose that $I\neq R$. Since $Spec_G(R)$ is a Noetherian space, then the subspace $V_G^R(I)$ is Noetherian by \cite[p. 123, Proposition 8(i)]{bourbakialgebre}. Since $I\neq R$, then $I$ has at least one minimal graded prime divisor by \cite[Corollary 2.3]{refai2004graded}. Using Theorem \ref{theorem 2.5}(2), $I$ contains only a finite number of minimal graded prime divisors $p_1, p_2, ..., p_n$, say, where $n\in \mathbb{Z}^{+}$. So $V_G^R(p_1), V_G^R(p_2), ..., V_G^R(p_n)$ are the only irreducible components of $V_G^R(I)$ by Theorem \ref{theorem 2.5}(1). Since every topological space is the union of its irreducible components, then $V_G^R(I)=V_G^R(p_1)\cup V_G^R(p_2)\cup ...\cup V_G^R(p_n)$. By \cite[Proposition 2.1(3)]{ozkiirisci2013graded}, we obtain $V_G^R(I)=V_G^R(\bigcap\limits_{i=1}^n p_i)$. By hypothesis and using Lemma \ref{Lemma 2.1}(3), we have $I=Gr(I)=Gr(\bigcap\limits_{i=1}^n p_i)=\bigcap\limits_{i=1}^n Gr(p_i)=\bigcap\limits_{i=1}^n p_i$, as needed.
\end{proof}
Recall that a graded ideal $I$ of a $G$-graded ring $R$ is said to be graded finitely generated if $I=Rr_1+Rr_2+...+Rr_n$ for some $r_1, r_2,..., r_n\in h(I)=I\cap h(R)$, see \cite{nastasescu2004methods}.
\begin{definition} We say that a graded ideal $I$ of a $G$-graded ring $R$ is an $RFG_g$-ideal if $Gr(I)=Gr(J)$ for some graded finitely generated ideal $J$ of $R$. In addition, we say that $R$ has property $(RFG_g)$ if every graded ideal $I$ of $R$ is an $RFG_g$-ideal.
\end{definition}
\begin{proposition}\label{proposition 2.9}
Let $R$ be a $G$-graded ring. Then:
\begin{enumerate}
\item If $I$ and $J$ are $RFG_g$-ideals of $R$, then so are $IJ$ and $I \cap J$.
\item If $I$ is an $RFG_g$-ideal of $R$, then $Gr(I)=Gr(Rx_1+Rx_2+...+Rx_n)$ for some $n\in\mathbb{Z}^+$ and $x_1, x_2,...,x_n\in h(I)$.
\end{enumerate}
\end{proposition}
\begin{proof}
(1) By hypothesis, $Gr(I)=Gr(I^\prime)$ and $Gr(J)=Gr(J^\prime)$ for some graded finitely generated ideals $I^\prime$ and $J^\prime$ of $R$. By \cite[Proposition 2.4(4)]{refai2000graded}, $Gr(IJ)=Gr(I\cap J)=Gr(I)\cap Gr(J)=Gr(I^\prime)\cap Gr(J^\prime)=Gr(I^\prime J^\prime)$. It is straightforward to see that $I^\prime J^\prime $ is a graded finitely generated ideal of $R$. Therefore, $IJ$ and $I\cap J$ are $RFG_g$-ideals of $R$.\\
(2) By hypothesis, $Gr(I)=Gr(T)$ for some graded finitely generated ideal $T$ of $R$ and hence $T=Ra_1+Ra_2+...+Ra_m$ for some $a_1, a_2, ..., a_m\in h(T)=T\cap h(R)$. Thus $Gr(I)=Gr(Ra_1+Ra_2+...+Ra_m)$. Since $a_i\in h(R)\cap Ra_i\subseteq h(R)\cap Gr(I)$ for each $i=1,...,m$, then $a_i\in \sqrt{I}$ for each i. This follows that for each $i=1,...,m$, $\exists k_i\in \mathbb{Z}^+$ such that $a_i^{k_i}\in I$. So $Ra_1^{k_1}+Ra_2^{k_2}+...+Ra_m^{k_m}\subseteq I$ and so $Gr(Ra_1^{k_1}+Ra_2^{k_2}+...+Ra_m^{k_m})\subseteq Gr(I)$. Now, let $p\in Spec_G(R)$ with $Ra_1^{k_1}+Ra_2^{k_2}+...+Ra_m^{k_m}\subseteq p$. Note that for each i, we have $a_i^{k_i}\in Ra_i\subseteq p$ and hence $a_i\in \sqrt{p}$. But $a_i\in h(R)$. This implies that $a_i\in Gr(p)=p$ as $p\in Spec_G(R)$. So $Ra_1+Ra_2+...+Ra_m\subseteq p$ and so $Gr(I)=Gr(Ra_1+Ra_2+...+Ra_m)\subseteq Gr(p)=p$. By Lemma \ref{Lemma 2.1}(2), we get $Gr(I)\subseteq \xi(V_G^R(Ra_1^{k_1}+Ra_2^{k_2}+...+Ra_m^{k_m}))=Gr(Ra_1^{k_1}+Ra_2^{k_2}+...+Ra_m^{k_m})$. Therefore $Gr(I)=Gr(Ra_1^{k_1}+Ra_2^{k_2}+...+Ra_m^{k_m})$. Choose $n=m$ and $x_i=a_i^{k_i}\in h(I)$ for each $i$, which completes the proof.
\end{proof}
Let $R$ be a $G$-graded ring. In \cite[Theorem 2.3]{ozkiirisci2013graded}, it has been proved that for each $r \in h(R)$, the set $D_r=Spec_G(R)-V_G^R(rR)$ is open in $Spec_G(R)$ and the family $\{D_r\, \mid \, r\in h(R)\}$ is a base for the Zariski topology on $Spec_G(R)$. In addition, $D_r$ is compact for each $r\in h(R)$. Now, we need the following lemma to prove the next theorem and it will also be used in the last section.
\begin{lemma} \cite[p. 123, Proposition 9]{bourbakialgebre} \label{Lemma 2.10} A topological space $X$ is Noetherian if and only if every open subset of $X$ is compact
\end{lemma}
\begin{theorem} \label{Theorem 2.11} A $G$-graded ring $R$ has Noetherian graded prime spectrum if and only if $R$ has property $(RFG_g)$.
\end{theorem}
\begin{proof} Suppose that $R$ has Noetherian graded prime spectrum and let $I\lhd_G R$. By Lemma \ref{Lemma 2.10}, the open set $Spec_G(R)-V_G^R(I)$ is compact. Since $\{D_r\,\mid\,r\in h(R)\}$ is a base for the Zariski topology on $Spec_G(R)$, then $Spec_G(R)-V_G^R(I)=\bigcup\limits_{i=1}^n D_{r_i}$ for some $r_1, r_2,..., r_n\in h(R)$. By \cite[Proposition 2.1(2)]{ozkiirisci2013graded}, we obtain $V_G^R(I)=Spec_G(R)-\bigcup\limits_{i=1}^n D_{r_i}=\bigcap\limits_{i=1}^n V_G^R(Rr_i)=V_G^R(\sum\limits_{i=1}^n Rr_i)$. By Lemma \ref{Lemma 2.1}(3), we get $Gr(I)=Gr(\sum\limits_{i=1}^n Rr_i)$ and hence $I$ is an $RFG_g$-ideal. Therefore, $R$ has property $(RFG_g)$. Conversely, suppose that $R$ has property $(RFG_g)$ and let $U=Spec_G(R)-V_G^R(K)$ be any open set in $Spec_G(R)$, where $K\lhd_G R$. By Lemma \ref{Lemma 2.10}, it is sufficient to prove that $U$ is compact. Note that $Gr(K)=Gr(\sum\limits_{i=1}^m Rx_i)$ for some $x_1, x_2, ..., x_m\in h(R)$. Again, using Lemma \ref{Lemma 2.1}(3) and \cite[Proposition 2.1(2)]{ozkiirisci2013graded}, we get $V_G^R(K)=\bigcap\limits_{i=1}^m V_G^R(Rx_i)$. Thus $U=Spec_G(R)-\bigcap\limits_{i=1}^m V_G^R(Rx_i)=\bigcup\limits_{i=1}^m D_{x_i}$ is a finite union of compact sets in $Spec_G(R)$ and hence $U$ is compact, as needed.
\end{proof}
\begin{proposition} \label{Proposititon 2.12} Let $R$ be a $G$-graded ring and $\Upsilon=\{I\lhd_G R\,\mid\, I$ is not an $RFG_g$-ideal of $R$ $\}$. If $\Upsilon\neq \emptyset$, then $\Upsilon$ contains maximal elements with respect to inclusion, and any such maximal element is graded prime ideal of $R$.
\end{proposition}
\begin{proof} Order $\Upsilon$ by inclusion, i.e., for $I_1, I_2\in \Upsilon$, $I_1\leq I_2$ if $I_1\subseteq I_2$. It is clear that $(\Upsilon, \leq)$ is a partially ordered set. Let $C=\{I_{\alpha}\,\mid\, \alpha\in\Delta\}$ be any non-empty chain subset of $\Upsilon$ and let $J=\underset{\alpha\in\Delta}{\bigcup} I_{\alpha}$. Clearly, $J\lhd_G R$. Now, assume by way of contradiction that $J\notin\Upsilon$. Then $J$ is an $RFG_g$-ideal of $R$. By Proposition \ref{proposition 2.9}(2), $Gr(J)=Gr(Rr_1+Rr_2+...+Rr_n)$ for some $r_1, r_2, ..., r_n\in h(J)=J\cap h(R)$. Hence, for each $i=1,...,n$, there exists $\alpha_i\in \Delta$ such that $r_i\in I_{\alpha_i}$. It is clear that the non-empty totally ordered set $\{I_{\alpha_1}, I_{\alpha_2},..., I_{\alpha_n}\}$ has maximum element, $I_{\alpha_t}$ say. Then $I_{\alpha_1}, I_{\alpha_2},..., I_{\alpha_n}\subseteq I_{\alpha_t}$ and hence $r_1, r_2,..., r_n\in I_{\alpha_t}$. Thus $Rr_1+Rr_2+...+Rr_n\subseteq I_{\alpha_t}\subseteq J$ which follows that $Gr(Rr_1+Rr_2+...+Rr_n)\subseteq Gr(I_{\alpha_t})\subseteq Gr(J)=Gr(Rr_1+Rr_2+...+Rr_n)$. So $Gr(I_{\alpha_t})=Gr(Rr_1+Rr_2+...+Rr_n)$ and so $I_{\alpha_t}$ is an $RFG_g$-ideal of $R$, a contradiction. Therefore $J\in \Upsilon$. But $J$ is an upper bound for $\Upsilon$. So by Zorn’s lemma, $\Upsilon$ has a maximal element $I$, say. Now, we show that $I\in Spec_G(R)$. If not, then $\exists A, B\lhd_G R$ such that $AB\subseteq I$ but $A\nsubseteq I$ and $B\nsubseteq I$ by \cite[Proposition 1.2]{refai2000graded}. This implies that $I\varsubsetneq A+I$, $I\varsubsetneq B+I$ and $(A+I)(B+I)\subseteq I\subseteq (A+I)\cap (B+I)$. Let $H=A+I$ and $K=B+I$. Then $H, K\lhd_G R$ with $I\varsubsetneq H$, $I\varsubsetneq K$ and $HK\subseteq I\subseteq H\cap K$. Since $I$ is a maximal element of $\Upsilon$, then $H$ and $K$ are $RFG_g$-ideals of $R$. But $Gr(HK)\subseteq Gr(I)\subseteq Gr(H\cap K)=Gr(HK)$ which follows that $Gr(HK)=Gr(I)$. By Proposition \ref{proposition 2.9}(1), $HK$ is an $RFG_g$-ideal of $R$ and hence $Gr(I)=Gr(HK)=Gr(L)$ for some graded finitely generated ideal $L$ of $R$. This implies that $I$ is an $RFG_g$-ideal, a contradiction. Therefore $I\in Spec_G(R)$, as desired.
\end{proof}
The following is an easy result of Proposition \ref{Proposititon 2.12} and Theorem \ref{Theorem 2.11}.
\begin{corollary} \label{Corollary 2.13}
A $G$-graded ring $R$ has Noetherian graded prime spectrum if and only if every graded prime ideal of $R$ is an $RFG_g$-ideal.
\end{corollary}
\section{Graded Zariski socles of graded submodules}
\begin{definition}
\begin{enumerate}
\item A $G$-graded $R$-module $M$ is said to be graded secondful if the natural map $\phi:Spec_G^s(M)\rightarrow Spec_G(\overline{R})$ defined by $S\rightarrow \overline{Ann_R(S)}$ is surjective.
\item Let $M$ be a $G$-graded $R$-module and $N\leq_G M$. The graded Zariski socle of $N$, denoted by $Z.soc_G(N)$, is the sum of all members of $V_G^s(N)$, i.e. $Z.soc_G(N)=\underset{S\in V_G^s(N)}{\sum}S$. If $V_G^s(N)=\emptyset$, then $Z.soc_G(N)=0$. Moreover, we say that a graded submodule $N$ of $M$ is a graded Zariski socle submodule if $N=Z.soc_G(N)$.
\end{enumerate}
\end{definition}
The Purpose of this section is to give some properties of the graded secondful modules and the graded Zariski socle of graded submodules. The results in this section are important in the last section.

In part (a) of the next example, we see that every non-zero graded module over a graded field is graded secondful. However, if $M$ is a graded secondful module over a graded ring $R$, then $R$ is not necessary a graded field and this will be discussed in part (b).
\begin{example} (a) Let $F$ be a $G$-graded field and $M$ be a non-zero $G$-graded $F$-module. Since the only graded ideals of $F$ are $\{0\}$ and $F$, then it is easy to see that $Ann_F(M)=\{0\}$, $Spec_G(\overline{F})=\{\overline{\{0\}}\}$ and $\overline{Ann_F(S)}=\overline{\{0\}}$ for each $S\in Spec_G^s(M)$. Hence $M$ is a graded secondful module.\\
(b) Let $R=\mathbb{Z}$, $G=\mathbb{Z}_2$ and $M=\mathbb{Z}_2\times\mathbb{Z}_2$. Then $R$ is a $G$-graded ring by $R_0=R$ and $R_1=\{0\}$. Also, $M$ is a $G$-graded $R$-module by $M_0=\{0\}\times\mathbb{Z}_2$ and $M_1=\mathbb{Z}_2\times\{0\}$. By some computations, we can see that $Spec_G^s(M)=\{M, M_0, M_1\}$, $Ann_R(M)=2\mathbb{Z}$ and $Spec_G(\overline{R})=Spec_G(\mathbb{Z}/2\mathbb{Z})=\{\overline{2\mathbb{Z}}\}$. Since $\overline{2\mathbb{Z}}=\overline{Ann_R(M_0)}=\phi(M_0)$, then $M$ is graded seconful. But $R$ is not a $Z_2$-graded field.
\end{example}
A proper graded ideal $J$ of a $G$-graded ring $R$ is called graded maximal if whenever $H$ is a $G$-graded ideal of $R$ with $J\subseteq H\subseteq R$, then either $H=J$ or $H=R$. The set of all graded maximal ideals of $R$ will be denoted by $Max_G(R)$. The graded Jacobson radical of $R$, denoted by $J_G(R)$, is the intersection of all graded maximal ideals of $R$, see \cite{nastasescu2004methods}.
\begin{proposition}\label{Proposition 3.3}
Let $M$ be a $G$-graded secondful $R$-module. Then the following hold:
\begin{enumerate}
\item If $M\neq 0$ and $I$ is a graded radical ideal of $R$, then $Ann_R(Ann_M(I))=I\Leftrightarrow Ann_R(M)\subseteq I$.
\item If $p\in Max_G(R)$ such that $Ann_M(p)=0$, then $\exists x\in p\cap R_e$ such that $(1+x)M=0$.
\item If $I$ is a graded radical ideal of $R$ contained in the graded Jacobson radical $J_G(R)$ such that $Ann_M(I)=0$, then $M=0$.
\item If $I$ is a graded radical ideal of $R$ such that $I\subseteq J(R_e)$ and $Ann_M(I)=0$, then $M=0$. Here, $J(R_e)$ is the Jacobson radical of the ring $R_e$.
\end{enumerate}
\end{proposition}
\begin{proof} (1) Since $Ann_M(I)\subseteq M$, then $Ann_R(M)\subseteq Ann_R(Ann_M(I))=I$. Conversely, it is clear that $I\subseteq Ann_R(Ann_M(I))$. Now, let $p\in Spec_G(R)$ with $I\subseteq p$. Then $Ann_R(M)\subseteq p$. Since $M$ is graded secondful, then $\exists S\in Spec_G^s(M)$ such that $Ann_R(S)=p$ which follows that $Ann_M(Ann_R(S))\subseteq Ann_M(I)$. Thus $Ann_R(Ann_M(I))\subseteq Ann_R(Ann_M(Ann_R(S)))=Ann_R(S)=p$ and hence $Ann_R(Ann_M(I))\subseteq p$ for each $p\in V_G^R(I)$. By Lemma \ref{Lemma 2.1}(2), we have $Ann_R(Ann_M(I))\subseteq \xi(V_G^R(I))=Gr(I)=I$, as needed.\\
(2) Note that $p\subseteq Ann_R(M)+p\subseteq R$. If $Ann_R(M)+p=p$, then $Ann_R(M)\subseteq p$ and hence $p=Ann_R(Ann_M(p))=Ann_R(\{0_M\})=R$ which is a contradiction. So $Ann_R(M)+p=R$ and so $1=r+i$ for some $r\in Ann_R(M)$ and $i\in p$. Since $1\in R_e$, then $1=1_e=(r+i)_e=r_e+i_e$ which implies that $r_e=1-i_e\in Ann_R(M)$ as $Ann_R(M)\lhd_G R$. Take $x=-i_e\in p\cap R_e$. Then $(1+x)M=0$.\\
(3)  Assume by way of contradiction that $M\neq 0$. Then $Ann_R(M)\neq R$ and hence $\exists p\in Max_G(R)$ such that $Ann_R(M)\subseteq p$ by \cite[Proposition 1.4]{refai2000graded}. By hypothesis, $I\subseteq p$ and thus $Ann_M(p)\subseteq Ann_M(I)=0$ which follows that $Ann_R(Ann_M(p))=R$. Using (1), we have $p=Ann_R(Ann_M(p))=R$, a contradiction.\\
(4) By \cite[Corollary 2.9.3]{nastasescu2004methods}, $ J(R_e)=J_G(R)\cap R_e$ and thus $I\subseteq J_G(R)$. Now, the result follows from (3).
\end{proof}
Let $M$ be a $G$-graded $R$-module and $Y$ be a subset of $Spec_G^s(M)$. Then the sum of all members of $Y$ will be expressed by $T(Y)$. If $Y=\emptyset$, we set $T(Y)=0$. It is clear that $T(Y_1\cup Y_2)=T(Y_1)+T(Y_2)$ for any $Y_1, Y_2\subseteq Spec_G^s(M)$. Also, if $Y_1\subseteq Y_2$, then $T(Y_1)\subseteq T(Y_2)$.
\begin{proposition}\label{Proposition 3.4} Let $M$ be a $G$-graded $R$-module and $N\leq_G M$. Then:
\begin{enumerate}
\item $soc_G(N)=T(V_G^{s*}(N))$.
\item $Z.soc_G(N)=T(V_G^s(N))$.
\item If $Y\subseteq Spec_G^s(M)$, then $Cl(Y)=V_G^s(T(Y))$. Therefore, $V_G^s(T(V_G^s(N))=V_G^s(Z.soc_G(N))=V_G^s(N)$.
\end{enumerate}
\end{proposition}
\begin{proof} (1) and (2) are clear and (3) has been shown in \cite[Proposition 4.1]{salam2022zariski}.
\end{proof}
In the next lemma, we give some situations where $V_G^s(N)$ and $V_G^{s*}(N)$ for a graded submodule $N$ of a $G$-graded module $M$ coincide.
\begin{lemma}\label{Lemma 3.5} Let $M$ be a $G$-graded $R$-module. If $N\leq_G M$ and $I\lhd_G R$, then:
\begin{enumerate}
\item $V_G^s(Ann_M(I))=V_G^s(Ann_M(Gr(I)))=V_G^{s*}(Ann_M(I))=V_G^{s*}(Ann_M(Gr(I)))$.
\item $V_G^s(N)=V_G^s(Ann_M(Ann_R(N)))=V_G^s(Ann_M(Gr(Ann_R(N))))=V_G^{s*}(Ann_M(Ann_R(N)))=V_G^{s*}(Ann_M(Gr(Ann_R(N))))$.
\end{enumerate}
\end{lemma}
\begin{proof} (1) It is straightforward to see that $V_G^s(Ann_M(J))=V_G^{s*}(Ann_M(J))$ for any $J\lhd_G R$. So it is enough to show that $V_G^{s*}(Ann_M(I))=V_G^{s*}(Ann_M(Gr(I))$. Since $I\subseteq Gr(I)$, then $Ann_M(Gr(I))\subseteq Ann_M(I)$ and hence $V_G^{s*}(Ann_M(Gr(I)))\subseteq V_G^{s*}(Ann_M(I))$. For the reverse inclusion, let $S\in V_G^{s*}(Ann_M(I))$. Then $S\subseteq Ann_M(I)$ and thus $I\subseteq Ann_R(S)$ which follows that $Gr(I)\subseteq Gr(Ann_R(S))=Ann_R(S)$ as $Ann_R(S)\in Spec_G(R)$. This implies that $S\subseteq Ann_M(Ann_R(S))\subseteq Ann_M(Gr(I))$. Therefore $S\in V_G^{s*}(Ann_M(Gr(I)))$, as desired. \\
(2) For any $S\in Spec_G^s(M)$, we have $S\in V_G^s(N)\Leftrightarrow Ann_R(Ann_M(Ann_R(N)))=Ann_R(N)\subseteq Ann_R(S)\Leftrightarrow S\in V_G^s(Ann_M(Ann_R(N)))$. This means that $V_G^s(N)=V_G^s(Ann_M(Ann_R(N)))$. Now, the result follows from (1).
\end{proof}
A $G$-graded $R$-module $M$ is said to be a comultiplication graded $R$-module if for every $N\leq_G M$, we have $N=Ann_M(I)$ for some $I\lhd_G R$ (see, for example, \cite{ansari2011graded,al2017some,abu2012comultiplication}). By \cite[Lemma 3.2]{abu2012comultiplication}, if $N$ is a graded submodule of a comultiplication graded $R$-module $M$, then $N=Ann_M(Ann_R(N))$.

In the next proposition, we compare the graded Zariski socle and the graded second socle of graded submodules. The proof of the proposition can be easily checked by using Proposition \ref{Proposition 3.4} and Lemma \ref{Lemma 3.5}.
\begin{proposition} \label{Proposition 3.6} Let $M$ be a $G$-graded $R$-module. Let $N, N^\prime\leq_G M$ and $I\lhd_G R$. Then the following hold:
\begin{enumerate}
\item $Z.soc_G(Ann_M(I))=Z.soc_G(Ann_M(Gr(I)))=soc_G(Ann_M(I))=soc_G(Ann_M(Gr(I)))$.
\item $Z.soc_G(N)=Z.soc_G(Ann_M(Ann_R(N)))=Z.soc_G(Ann_M(Gr(Ann_R(N))))=soc_G(Ann_M(Ann_R(N)))=soc_G(Ann_M(Gr(Ann_R(N))))$.
\item $soc_G(N)\subseteq Z.soc_G(N)$. Moreover, if $M$ is a comultiplication $G$-graded $R$-module, then the equality holds.
\item If $V_G^s(N)\subseteq V_G^s(N^\prime)$, then $Z.soc_G(N)\subseteq Z.soc_G(N^\prime)$.
\item $V_G^s(N)=V_G^s(N^\prime)\Leftrightarrow Z.soc_G(N)=Z.soc_G(N^\prime)$.
\end{enumerate}
\end{proposition}
The following example shows that the reverse inclusion in Proposition 3.6(3) is not true in general.
\begin{example} Let $F=\mathbb{R}$ (The field of real numbers), $G=\mathbb{Z}_2$ and $M=\mathbb{R}\times\mathbb{R}$. Then $F$ is a $G$-graded field by $F_0=F$ and $F_1=\{0\}$. Also $M$ is a $G$-graded $F$-module by $M_0=\mathbb{R}\times \{0\}$ and $M_1=\{0\}\times \mathbb{R}$. Clearly, $M_0\leq_G M$. By \cite[Lemma 2.8(1)]{salam2022zariski}, every non-zero graded submodule of a graded module over a graded field is graded second, which follows that $M_0\in Spec_G^s(M)$ and hence $soc_G(M_0)=M_0$. Since $M_0\neq 0$, then $Ann_F(M_0)\neq F$ and thus $Ann_F(M_0)=0$. So we get $V_G^s(M_0)=Spec_G^s(M)$. Again, by \cite[Lemma 2.8(1)]{salam2022zariski}, $M\in Spec_G^s(M)$. Consequently, $Z.soc_G(M_0)=T(V_G^s(M_0))=T(Spec_G^s(M))=M\neq M_0$.
\end{example}
In the following proposition, we give some more properties of both $V_G^s(N)$ and $Z.soc_G(N)$ for $N\leq_G M$.
\begin{proposition} \label{Proposition 3.8} Suppose that $N$ and $N^\prime$ are graded submodules of a $G$-graded $R$-module $M$. Then the following hold:
\begin{enumerate}
\item[(a)] $Z.soc_G(0)=0$.
\item[(b)] If $N\subseteq N^\prime$, then $Z.soc_G(N)\subseteq Z.soc_G(N^\prime)$.
\item[(c)] $Z.soc_G(Z.soc_G(N))=Z.soc_G(N)$.
\item[(d)] $Z.soc_G(N+N^\prime)=Z.soc_G(N)+Z.soc_G(N^\prime)$.
\item[(e)] If $M$ is graded secondful, then $N\neq 0\Leftrightarrow V_G^s(N)\neq \emptyset \Leftrightarrow Z.soc_G(N)\neq 0$.
\item[(f)] $Gr(Ann_R(N))\subseteq Ann_R(Z.soc_G(N))$. If $M$ is graded secondful, then the equality holds.
\end{enumerate}
\end{proposition}
\begin{proof} (a) $Z.soc_G(0)=T(V_G^s(0))=T(\emptyset)=0$.\\
(b) Since $N\subseteq N^\prime$, then $Ann_R(N)\subseteq Ann_R(N^\prime)$ and hence $V_G^s(N)\subseteq V_G^s(N^\prime)$. By Proposition \ref{Proposition 3.6}(4), we get $Z.soc_G(N)\subseteq Z.soc_G(N^\prime)$.\\
(c) By Lemma \ref{Lemma 3.5}(3), we have $V_G^s(Z.soc_G(N))=V_G^s(N)$. So $Z.soc_G(Z.soc_G(N))=Z.soc_G(N)$ by Proposition \ref{Proposition 3.6}(5).\\
(d) By \cite[Theorem 2.16(3)]{salam2022zariski}, we have $V_G^s(N+N^\prime)=V_G^s(N)\cup V_G^s(N^\prime)$ which follows that $Z.soc_G(N+N^\prime)=T(V_G^s(N+N^\prime))=T(V_G^s(N)\cup V_G^s(N^\prime))=T(V_G^s(N))+T(V_G^s(N^\prime))=Z.soc_G(N)+Z.soc_G(N^\prime)$.\\
(e) By part (a) and Proposition \ref{Proposition 3.6}(5), $V_G^s(N)\neq \emptyset\Leftrightarrow V_G^s(N)\neq V_G^s(0)\Leftrightarrow Z.soc_G(N)\neq Z.soc_G(0)=0$. Now, it is clear that $V_G^s(N)\neq \emptyset\Rightarrow N\neq 0$ and it remains to prove the converse. So suppose that $N\neq 0$. Then $Ann_R(N)\neq R$. By \cite[Proposition 1.4]{refai2000graded}, we obtain $Ann_R(N)\subseteq p$ for some $p\in Max_G(R)\subseteq Spec_G(R)$. By hypothesis, $\exists S\in Spec_G^s(M)$ such that $Ann_R(S)=p$ which follows that $Ann_R(N)\subseteq Ann_R(S)$, i.e. $S\in V_G^s(N)$.\\
(f) By Proposition \ref{Proposition 3.6}(2), $Z.soc_G(N)=Soc_G(Ann_M(Gr(Ann_R(N))))\subseteq Ann_M(Gr(Ann_R(N)))$ and thus $Gr(Ann_R(N))\subseteq Ann_R(Z.soc_G(N))$. Now, suppose that $M$ is graded secondful and we show that $Ann_R(Z.soc_G(N))\subseteq Gr(Ann_R(N))$. If $V_G^s(N)=\emptyset$, then $N=0$ by (e). Hence $Gr(Ann_R(N))=Gr(R)=R$  and the result is trivially true. So suppose that $V_G^s(N)\neq \emptyset$. So $N\neq 0$ and so $Ann_R(N)\neq R$ which follows that $V_G^R(Ann_R(N))\neq \emptyset$. So let $p\in V_G^R(Ann_R(N))$. Then $\exists S\in Spec_G^s(M)$ such that $Ann_R(S)=p$ and hence $Ann_R(N)\subseteq Ann_R(S)$. This implies that $Ann_R(Z.soc_G(N))=Ann_R(\underset{H\in V_G^s(N)}{\sum}H)=\underset{H\in V_G^s(N)}{\bigcap}Ann_R(H)\subseteq Ann_R(S)=p$. Therefore, $Ann_R(Z.soc_G(N))\subseteq p$ for each $p\in V_G^R(Ann_R(N))$. By Lemma \ref{Lemma 2.1}(2), we have $Ann_R(Z.soc_G(N))\subseteq \xi(V_G^R(Ann_R(N)))=Gr(Ann_R(N))$ which completes the proof.
\end{proof}
\section{Graded modules with Noetherian graded second spectrum}
Let $M$ be a $G$-graded $R$-module. In this section, we investigate $Spec_G^s(M)$ with respect to Zariski topology from the viewpoint of being a Noetherian space.
\begin{theorem} \label{Theorem 4.1}
A $G$-graded $R$-module $M$ has Noetherian graded second spectrum if and only if the descending chain condition for graded Zariski socle submodules of $M$ hold.
\end{theorem}
\begin{proof} Suppose that $Spec_G^s(M)$ is a Noetherian space and let $N_1\supseteq N_2\supseteq...$ be a descending chain of graded Zariski socle submodules of $M$, where $N_i\leq_G M$. Then $V_G^s(N_1)\supseteq V_G^s(N_2)\supseteq...$ is a descending chain of closed sets in $Spec_G^s(M)$. So $\exists k\in \mathbb{Z}^{+}$ such that $V_G^s(N_i)=V_G^s(N_k)$ for each $i\geq k$. By Proposition \ref{Proposition 3.6}(5), we have $N_k=Z.soc_G(N_k)=Z.soc_G(N_i)=N_i$ for each $i\geq k$. Conversely, suppose that the descending chain condition for graded Zariski socle submodules of $M$ hold and let $V_G^s(N_1)\supseteq V_G^s(N_2)\supseteq ...$ be a descending chain of closed sets in $Spec_G^s(M)$, where $N_i\leq_G M$. Then $T(V_G^s(N_1))\supseteq T(V_G^s(N_2))\supseteq ...$. By Proposition \ref{Proposition 3.4}(2) and Proposition \ref{Proposition 3.8}(c), we have $Z.soc_G(N_1)\supseteq Z.soc_G(N_2)\supseteq ...$ which is a descending chain of graded Zariski socle submodules of $M$. Therefore, $\exists k\in \mathbb{Z}^{+}$ such that $Z.soc_G(N_k)=Z.soc_G(N_i)$ for each $i\geq k$. Consequently, $V_G^s(N_k)=V_G^s(N_i)$ for each $i\geq k$ by Proposition \ref{Proposition 3.6}(5) again.
\end{proof}
The proof of the following result is clear by Theorem \ref{Theorem 4.1}.
\begin{corollary}
Every graded Artinian module has Noetherian graded second spectrum.
\end{corollary}
In the following lemma, we recall some properties of the natural map $\phi$ of $Spec_G^s(M)$. These properties are important in the rest of this section.
\begin{lemma} (\cite[Proposition 3.1 and Theorem 3.7]{salam2022zariski}) \label{Lemma 4.3}
Let $M$ be a $G$-graded $R$-module. Then the following hold:
\begin{enumerate}
\item $\phi$ is continuous and $\phi^{-1}(V_G^{\overline{R}}(\overline{I}))=V_G^s(Ann_M(I))$ for every graded ideal $I$ of $R$ containing $Ann_R(M)$.
\item If $M$ is graded secondful, then $\phi$ is closed with $\phi(V_G^s(N))=V_G^{\overline{R}}(\overline{Ann_R(N)})$ for any $N\leq_G M$.
\end{enumerate}
\end{lemma}
Now, we need the following lemma to prove the next theorem.
\begin{lemma} \label{Lemma 4.4}
If $M$ is a non-zero $G$-graded secondful $R$-module, then $V_G^{\overline{R}}(\overline{Ann_R(Ann_M(I))})=V_G^{\overline{R}}(\overline{I})$ for every graded ideal $I$ of $R$ containing $Ann_R(M)$.
\end{lemma}
\begin{proof} By Lemma \ref{Lemma 3.5}(1), we have $V_G^s(Ann_M(Gr(I)))=V_G^s(Ann_M(I))$. Since $M$ is graded secondful, then $V_G^{\overline{R}}(\overline{Ann_R(Ann_M(Gr(I)))})=\phi(V_G^s(Ann_M(Gr(I))))=\phi(V_G^s(Ann_M(I)))=V_G^{\overline{R}}(\overline{Ann_R(Ann_M(I))})$. By Proposition \ref{Proposition 3.3}(1), we have $V_G^{\overline{R}}(\overline{Gr(I)})=V_G^{\overline{R}}(\overline{Ann_R(Ann_M(I))})$. But it is easy to see that $V_G^{\overline{R}}(\overline{Gr(I)})=V_G^{\overline{R}}(\overline{I})$. Consequently, $V_G^{\overline{R}}(\overline{I})=V_G^{\overline{R}}(\overline{Ann_R(Ann_M(I))})$.
\end{proof}
In Lemma \ref{Lemma 4.4}, if we drop the condition that $M$ is a graded secondful module, then the equality might not happen. For this, take the ring of integers $R=\mathbb{Z}$ as $\mathbb{Z}_2$-graded $\mathbb{Z}$-module by $R_0=\mathbb{Z}$ and $R_1=\{0\}$. Note that $Ann_{\mathbb{Z}}(\mathbb{Z})=\{0\}$. By \cite{salam2022zariski}, $Spec_{\mathbb{Z}_2}^s(\mathbb{Z})=\emptyset$. Now, it is clear that $\mathbb{Z}$ is not a $\mathbb{Z}_2$-graded secondful $\mathbb{Z}$-module. Let $I=2\mathbb{Z}$. Then $V_{\mathbb{Z}_2}^{\overline{\mathbb{Z}}}(\overline{Ann_{\mathbb{Z}}(Ann_{\mathbb{Z}}(I))})=\emptyset$. But $\overline{I}\in V_{\mathbb{Z}_2}^{\overline{\mathbb{Z}}}(\overline{I})$. So $V_{\mathbb{Z}_2}^{\overline{\mathbb{Z}}}(\overline{I})\neq V_{\mathbb{Z}_2}^{\overline{\mathbb{Z}}}(\overline{Ann_{\mathbb{Z}}(Ann_{\mathbb{Z}}(I))})$.
\begin{theorem}\label{Theorem 4.5}
Let $M$ be a $G$-graded secondful $R$-module. Then $M$ has Noetherian graded second spectrum if and only if $\overline{R}$ has Noetherian graded prime spectrum.
\end{theorem}
\begin{proof}If $M=0$, then the result is trivially true. So assume that $M\neq 0$. Suppose that $M$ has Noetherian graded second spectrum and let $V_G^{\overline{R}}(\overline{I_1})\supseteq V_G^{\overline{R}}(\overline{I_2})\supseteq...$ be a descending chain of closed sets in $Spec_G(\overline{R})$, where $I_t$ is a $G$-graded ideal of $R$ containing $Ann_R(M)$ for each $t$. Then $\phi^{-1}(V_G^{\overline{R}}(\overline{I_1}))\supseteq \phi^{-1}(V_G^{\overline{R}}(\overline{I_2}))\supseteq...$ and hence, by Lemma \ref{Lemma 4.3}(1), $V_G^s(Ann_M(I_1))\supseteq V_G^s(Ann_M(I_2))\supseteq ...$ which is a descending chain of closed sets in $Spec_G^s(M)$. Thus $\exists k\in \mathbb{Z}^{+}$ such that $V_G^s(Ann_M(I_k))=V_G^s(Ann_M(I_h))$ for each $h\geq k$ and thus $V_G^{\overline{R}}(\overline{Ann_R(Ann_M(I_k))})=\phi(V_G^s(Ann_M(I_k)))=\phi(V_G^s(Ann_M(I_h)))=V_G^{\overline{R}}(\overline{Ann_R(Ann_M(I_h))})$ by Lemma \ref{Lemma 4.3}(2). Now, using Lemma \ref{Lemma 4.4}, we get $V_G^{\overline{R}}(\overline{I_k})=V_G^{\overline{R}}(\overline{I_h})$ for each $h\geq k$ which completes the proof of the first direction. For the converse, suppose that $\overline{R}$ has Noetherian graded prime spectrum and let $V_G^s(N_1)\supseteq V_G^s(N_2)\supseteq ...$ be a descending chain of closed sets in $Spec_G^s(M)$, where $N_i\leq_G M$. So $\phi(V_G^s(N_1))\supseteq \phi(V_G^s(N_2))\supseteq...$ and so $V_G^{\overline{R}}(\overline{Ann_R(N_1)})\supseteq V_G^{\overline{R}}(\overline{Ann_R(N_2)})\supseteq ...$ is a descending chain of closed sets in $Spec_G(\overline{R})$ by Lemma \ref{Lemma 4.3}(2). This follows that $\exists k\in \mathbb{Z}^{+}$ such that $V_G^{\overline{R}}(\overline{Ann_R(N_i)})=V_G^{\overline{R}}(\overline{Ann_R(N_k)})$ for each $i\geq k$ and hence $\phi^{-1}(V_G^{\overline{R}}(\overline{Ann_R(N_i)}))=\phi^{-1}(V_G^{\overline{R}}(\overline{Ann_R(N_k)}))$. By Lemma \ref{Lemma 3.5}(2) and Lemma \ref{Lemma 4.3}(1), we have $V_G^s(N_i)=V_G^s(Ann_M(Ann_R(N_i)))=V_G^s(Ann_M(Ann_R(N_k)))=V_G^s(N_k)$ for each $i\geq k$.
\end{proof}
\begin{corollary} Let $M$ be a $G$-graded secondful $R$-module. Then we have the following:
\begin{enumerate}
\item[(i)] If $R$ has Noetherian graded prime spectrum, then $M$ has Noetherian graded second spectrum.
\item[(ii)] If $R$ is graded Noetherian ring, then $M$ has Noetherian graded second spectrum.
\end{enumerate}
\end{corollary}
\begin{proof} (i) It is clear by Lemma \ref{Lemma 2.1}(4) and Theorem \ref{Theorem 4.5}.\\
(ii) Since $R$ is graded Noetherian ring, then it has Noetherian graded prime spectrum by Proposition \ref{Proposition 2.2}. Now, the result follows from part (i).
\end{proof}
In the following lemma, we give a property for the graded second socle of graded submodules which will be used in the proof of the next theorem.
\begin{lemma} \label{lemma 4.7}
Let $M$ be a $G$-graded $R$-module and $I_1, I_2\lhd_G R$. Then $soc_G(Ann_M(I_1 I_2))=soc_G(Ann_M(I_1))+soc_G(Ann_M(I_2))$.
\end{lemma}
\begin{proof} By \cite[Corollary 2.14]{salam2022zariski}, we have $V_G^{s*}(Ann_M(I_1))\cup V_G^{s*}(Ann_M(I_2))=V_G^{s*}(Ann_M(I_1 I_2))$. This follows that $soc_G(Ann_M(I_1 I_2))=T(V_G^{s*}(Ann_M(I_1 I_2)))=T(V_G^{s*}(Ann_M(I_1))\cup V_G^{s*}(Ann_M(I_2)))=T(V_G^{s*}(Ann_M(I_1)))+T(V_G^{s*}(Ann_M(I_2)))=soc_G(Ann_M(I_1))+soc_G(Ann_M(I_2))$.
\end{proof}
A $G$-graded $R$-module $M$ is said to be a graded weak comultiplication module if $Spec_G^s(M)=\emptyset$ or any graded second submodule $S$ of $M$ has the form $S=Ann_M(I)$ for some graded ideal $I$ of $R$. It can be easily checked that a $G$-graded $R$-module $M$ is a graded weak comultiplication module if and only if $S=Ann_M(Ann_R(S))$ for each graded second submodule $S$ of $M$.
\begin{theorem} \label{theorem 4.8}
Let $R$ be a $G$-graded ring with Noetherian graded prime spectrum. If $M$ is a non-zero graded secondful weak comultiplication $R$-module, then the following hold:
\begin{enumerate}
\item Any graded Zariski socle submodule $N$ of $M$ is the sum of a finite number of graded seccond submoduels.
\item A graded submodule $N$ of $M$ is a graded Zariski socle submodule of $M$ if and only if $N=0$ or $N=\sum\limits_{i=1}^n Ann_M(p_i)$ for some $p_i\in V_G^R(Ann_R(N))$ and $n\in \mathbb{Z}^{+}$.
\end{enumerate}
\end{theorem}
\begin{proof} (1) Note that $\{0\}$ is the sum of the empty family of graded second submodules of $M$. Suppose that $N\neq 0$. Then $Ann_R(N)\neq R$ and hence $Gr(Ann_R(N))\neq R$. By Theorem \ref{Theorem 2.7}, we have $Gr(Ann_R(N))$ is the intersection of a finite number of minimal graded prime divisors of it. So $\exists n\in \mathbb{Z}^{+}$ and $p_1, p_2, ..., p_n$ minimal graded prime divisors of $Gr(Ann_R(N))$ such that $Gr(Ann_R(N))=\bigcap\limits_{i=1}^n p_i$. Since $M$ is graded secondful, then for each $i=1, ..., n$, there exists $S_i\in Spec_G^s(M)$ such that $p_i=Ann_R(S_i)$. As $M$ is graded weak comultiplication, then $S_i=Ann_M(Ann_R(S_i))=Ann_M(p_i)$ for each $i$. By Proposition \ref{Proposition 3.6}(2), \cite[Proposition 2.4]{refai2000graded} and Lemma \ref{lemma 4.7}, we obtain $N=Z.soc_G(N)=soc_G(Ann_M(Gr(Ann_R(N))))=soc_G(Ann_M(\bigcap\limits_{i=1}^n p_i))=soc_G(Ann_M(\bigcap\limits_{i=1}^n Gr(p_i)))=soc_G(Ann_M(Gr(p_1 p_2...p_n)))=soc_G(Ann_M(p_1 p_2...p_n))=\sum\limits_{i=1}^n soc_G(Ann_M(p_i))=\sum\limits_{i=1}^n S_i$.\\
(2) $\Rightarrow$: The proof is clear by part (1). \\
$\Leftarrow$: Suppose that $N=\sum\limits_{i=1}^n Ann_M(p_i)$ for some $p_i\in V_G^R(Ann_R(N))$ and $n\in \mathbb{Z}^{+}$. Note that for each $i=1,...,n$, there exists $S_i\in Spec_G^s(M)$ such that $p_i=Ann_R(S_i)$ as $M$ is a graded secondful module. Since $M$ is a graded weak comultiplication module, then $Ann_M(p_i)=Ann_M(Ann_R(S_i))=S_i\in Spec_G^s(M)$ for each $i$. This follows that $N=\sum\limits_{i=1}^n Ann_M(p_i)=\sum\limits_{i=1}^n soc_G(Ann_M(p_i))=\sum\limits_{i=1}^n Z.soc_G(Ann_M(p_i))=Z.soc_G(\sum\limits_{i=1}^n Ann_M(p_i))=Z.soc_G(N)$ by Proposition \ref{Proposition 3.6}(1) and Proposition \ref{Proposition 3.8}(d). Therefore $N=Z.soc_G(N)$, as needed.
\end{proof}
\begin{definition} A graded submodule $N$ of a $G$-graded $R$-module $M$ is said to be an $RFG_g^{*}$-submodule if $Z.soc_G(N)=Z.soc_G(Ann_M(I))$ for some graded finitely generated ideal $I$ of $R$. In addition, we say that $M$ has property $(RFG_g^{*})$ if every graded submodule of $M$ is an $RFG_g^{*}$-submodule.
\end{definition}
By Proposition \ref{Proposition 3.6}(2), every graded module over a graded Noetherian ring has property $(RFG_g^{*})$. \\
Let $M$ be a $G$-graded $R$-module. In \cite[Proposition 3.13 and Theorem 3.15]{salam2022zariski}, we have proved that for each $r\in h(R)$, the set $X_r^s=Spec_G^s(M)-V_G^s(Ann_M(r))$ is open in $Spec_G^s(M)$ and the family $\{X_r^s\,\mid\, r\in h(R)\}$ is a base for the Zariski topology on $Spec_G^s(M)$. In addition, if $M$ is a graded secondful module, then $X_r^s$ is compact for each $r\in h(R)$.
\begin{theorem} \label{theorem 2.10}
Let $M$ be a $G$-graded secondful $R$-module. Then $M$ has Noetherian graded second spectrum if and only if $M$ has property $(RFG_g^{*})$.
\end{theorem}
\begin{proof}
Suppose that $M$ has Noetherian graded second spectrum and let $N\leq_G M$. By Lemma \ref{Lemma 2.10}, the open set $Spec_G^s(M)-V_G^s(N)$ is compact. Since $\{X_r^s\,\mid\, r\in h(R)\}$ is a base for the Zariski topology on $Spec_G^s(M)$, then $Spec_G^s(M)-V_G^s(N)=\bigcup\limits_{i=1}^n X_{r_i}^s$ for some $r_1,r_2, ..., r_n\in h(R)$. Using \cite[Theorem 2.16(2)]{salam2022zariski}, $V_G^s(M)=Spec_G^s(M)-\bigcup\limits_{i=1}^n X_{r_i}^s=\bigcap\limits_{i=1}^n V_G^s(Ann_M(r_i))=V_G^s(\bigcap\limits_{i=1}^n Ann_M(Ann_R(Ann_M(r_i))))=V_G^s(\bigcap\limits_{i=1}^n Ann_M(r_i))=V_G^s(Ann_M(\sum\limits_{i=1}^n Rr_i))$. By Proposition \ref{Proposition 3.6}(5), we obtain $Z.soc_G(N)=Z.soc_G(Ann_M(\sum\limits_{i=1}^n Rr_i))$ which follows that $N$ is an $RFG_g^{*}$-submodule. Consequently, $M$ has property $(RFG_g^{*})$. Conversely, assume that $M$ has property $(RFG_g^{*})$ and let $U=Spec_G^s(M)-V_G^s(T)$ be any open set in $Spec_G^s(M)$, where $N\leq_G M$. By Lemma \ref{Lemma 2.10}, it is enough to show that $U$ is compact. Remark that $Z.soc_G(T)=Z.soc_G(Ann_M(\sum\limits_{i=1}^k Rr_i))$ for some $r_1, r_2, ..., r_k\in h(R)$. Again, by Proposition \ref{Proposition 3.6}(5) and \cite[Theorem 2.16(2)]{salam2022zariski}, we have $V_G^s(T)=\bigcap\limits_{i=1}^k V_G^s(Ann_M(r_i))$. Hence $U=\bigcup\limits_{i=1}^k X_{r_i}^s$ is a finite union of compact sets in $Spec_G^s(M)$ and thus $U$ is compact, as desired.
\end{proof}
A $G$-graded $R$-module $M$ is said to be graded faithful if whenever $r\in h(R)$ with $rM=\{0_M\}$ implies that $r=0$. In other words, $M$ is a $G$-graded faithful $R$-module if the set $Ann_{h(R)}(M)=\{r\in h(R)\,\mid\, rM=\{0_M\}\}=\{0_R\}$, see \cite{oral2011graded}. Let $M$ be a $G$-graded $R$-module and $I\lhd_G R$. It is easy to see that if $Ann_{h(R)}(M)\subseteq I$, then $Ann_R(M)\subseteq I$.
\begin{lemma} \label{lemma 4.11}
Let $M$ be a $G$-graded seconful faithful $R$-module and $N\leq_G M$. Then $N$ is an $RFG_g^{*}$-submodule of $M$ if and only if $Ann_R(N)$ is an $RFG_g$-ideal.
\end{lemma}
\begin{proof} $\Rightarrow$: Suppose that $N$ is an $RFG_g^{*}$-submodule. Then $Z.soc_G(N)=Z.soc_G(Ann_M(I))$ for some graded finitely generated ideal $I$ of $R$. Since $M$ is graded secondful, by Proposition \ref{Proposition 3.8}(f) and Proposition \ref{Proposition 3.6}(1), we have $Gr(Ann_R(N))=Ann_R(Z.soc_G(N))=Ann_R(Z.soc_G(Ann_M(I)))=Ann_R(Z.soc_G(Ann_M(Gr(I))))=Gr(Ann_R(Ann_M(Gr(I))))$. Since $M$ is a graded faithful module, then $Ann_{h(R)}(M)=\{0\}\subseteq I$ and hence $Ann_R(M)\subseteq I\subseteq Gr(I)$. By Proposition \ref{Proposition 3.3}(1), we get $Ann_R(Ann_M(Gr(I)))=Gr(I)$ and thus $Gr(Ann_R(N))=Gr(Gr(I))=Gr(I)$. Therefore $Ann_R(N)$ is an $RFG_g$-ideal of $R$.\\
$\Leftarrow$: Suppose that $Ann_R(N)$ is an $RFG_g$-ideal of $R$. Then $Gr(Ann_R(N))=Gr(J)$ for some graded finitely generated ideal $J$ of $R$. By Proposition \ref{Proposition 3.6}, we obtain $Z.soc_G(N)=Z.soc_G(Ann_M(Gr(Ann_R(N))))=Z.soc_G(Ann_M(Gr(J)))=Z.soc_G(Ann_M(J))$ and thus $N$ is an $RFG_g^{*}$-submodule.
\end{proof}
\begin{corollary} \label{corollary 4.12}
Let $M$ be a $G$-graded secondful faithful $R$-module. Then the following hold:
\begin{enumerate}
\item $M$ has Noetherian graded second spectrum if and only if every graded second submodule of $M$ is an $RFG_g^{*}$-submodule.
\item If $N_1$ and $N_2$ are $RFG_g^{*}$-submodules of $M$, then so is $N_1+ N_2$.
\end{enumerate}
\end{corollary}
\begin{proof} (1) $\Rightarrow$: It is obvious by Theorem \ref{theorem 2.10}.\\
$\Leftarrow$: Suppose that every graded second submodule of $M$ is an $RFG_g^{*}$-submodule and let $p\in Spec_G(R)$. Since $M$ is graded faithful, then $Ann_{h(R)}(M)=\{0\}\subseteq p$ and hence $Ann_R(M)\subseteq p$. As $M$ is graded secondful, then $\exists S\in Spec_G^s(M)$ such that $Ann_R(S)=p$. So $S$ is an $RFG_g^{*}$-submodule and so $p$ is an $RFG_g$-ideal of $R$ by Lemma \ref{lemma 4.11}. This means that every graded prime ideal of $R$ is an $RFG_g$-ideal. Thus $R$ has Noetherian graded prime spectrum by Corollary \ref{Corollary 2.13}. This implies that $\overline{R}$ has Noetherian graded prime spectrum using Lemma \ref{Lemma 2.1}(4). By Theorem \ref{Theorem 4.5}, we obtain $M$ has Noetherian graded second spectrum, as needed.\\
(2) By Lemma \ref{lemma 4.11}, $Ann_R(N_1)$ and $Ann_R(N_2)$ are $RFG_g$-ideals of $R$. But $Ann_R(N_1+N_2)=Ann_R(N_1)\cap Ann_R(N_2)$. Using Proposition \ref{proposition 2.9}(1), we have $Ann_R(N_1+N_2)$ is an $RFG_g$-ideal of $R$. Again, by Lemma \ref{lemma 4.11}, we obtain $N_1+N_2$ is an $RFG_g^{*}$-submodule of $M$.
\end{proof}

\bigskip\bigskip\bigskip\bigskip

\end{document}